\documentclass{amsart}
\usepackage{amsmath}
\usepackage{amssymb, amsthm, amsfonts, tikz-cd, mathrsfs,mathtools,stmaryrd,enumitem, algorithmic,float,comment,tikz}
\usetikzlibrary{backgrounds}
\usetikzlibrary{decorations.pathreplacing}
\usepackage{fullpage}
\usepackage{xcolor}
\usepackage[ruled,vlined,linesnumbered]{algorithm2e}
\usepackage[colorlinks=true,allcolors=black]{hyperref}

\usepackage{comment}

\usepackage[capitalize,nameinlink]{cleveref}
\crefname{theorem}{Theorem}{Theorems}
\crefname{lemma}{Lemma}{Lemmas}
\crefname{claim}{Claim}{Claims}
\crefname{prop}{Proposition}{Propositions}

\newtheorem{theorem}{Theorem}

\newtheorem{claim}[theorem]{Claim}
\newtheorem{corollary}[theorem]{Corollary}

\newtheorem*{claim*}{Claim}

\theoremstyle{remark}
\newtheorem*{remark*}{Remark}
\newtheorem{remark}[theorem]{Remark}

\numberwithin{theorem}{section}

\renewcommand{\phi}{\varphi}

\newcommand{\cH}{\mathcal H}

\newcommand{\cS}{\mathcal S}

\def\1{\mathbbm{1}}

\renewcommand{\le}{\leqslant}
\renewcommand{\ge}{\geqslant}

\newcommand{\R}{\mathbb R}

\newcommand{\floor}[1]{\lfloor#1\rfloor}
\newcommand{\ang}[1]{\langle#1\rangle}

\newcommand{\bin}{\{0,1\}}
\newcommand{\supp}{\textup{supp}}

\usepackage[
backend=biber,
style=numeric,
maxnames=99,
giveninits=true
]{biblatex}
\renewbibmacro{in:}{%
  \ifentrytype{article}{}{\printtext{\bibstring{in}\intitlepunct}}}

\title
{Nondegenerate hyperplane covers of the hypercube} 

\author{Lisa Sauermann}
\thanks{Institute for Applied Mathematics, University of Bonn, Germany. Email: \texttt{sauermann@iam.uni-bonn.de}. Supported by the DFG Heisenberg Program.}

\author{Zixuan Xu}
\thanks{Massachusetts Institute of Technology. Email: \texttt{zixuanxu@mit.edu}}

\begin{document}

\begin{abstract}
We consider collections of hyperplanes in $\mathbb{R}^n$ covering all vertices of the $n$-dimensional hypercube $\{0,1\}^n$, which satisfy the following nondegeneracy condition: For every $v\in \{0,1\}^n$ and every $i=1,\dots,n$, we demand that there is a hyperplane $H$ in the collection with $v\in H$ such that the variable $x_i$ appears with a non-zero coefficient in the hyperplane equation describing $H$. We prove that every collection $\mathcal{H}$ of hyperplanes in $\mathbb{R}^n$ covering $\{0,1\}^n$ with this nondegeneracy condition must have size $|\mathcal{H}|\ge n/2$.

This bound is tight up to constant factors. It generalizes a recent result concerning the intensively studied skew covers problem, which asks about the minimum possible size of a hyperplane cover of $\{0,1\}^n$ in which all variables appear with non-zero coefficients in all hyperplane equations.

As an application of our result, we also obtain an essentially tight bound for an old problem about collections of hyperplanes slicing all edges of the $n$-dimensional hypercube, in the case where all of the hyperplanes have bounded integer coefficients.

\end{abstract}
		
\maketitle

\vspace{-2em}

\section{Introduction}

A collection of (affine) hyperplanes $\cH$ in $\R^n$ is said to cover the vertices of the $n$-dimensional hypercube $\bin^n$ if for every vertex $v\in \bin^n$, there exists a hyperplane $H\in \cH$ such that $v\in H$. It is natural to ask about the minimum possible size of a hyperplane collection $\cH$ satisfying this condition, i.e., to ask how many hyperplanes are needed to cover the $n$-dimensional hypercube. It turns out that without any further restrictions the answer to this question is trivial: Two hyperplanes suffice, for example one can take the two hyperplanes given by the equations $x_1 = 0$ and $x_1 = 1$ (and it is easy to see that one hyperplane is not enough to cover all vertices of the hypercube).

However, questions of this type with various nondegeneracy conditions for $\cH$ have been studied intensively in the literature, see for example \cite{Aaronson-et-al-21,Araujo-Balogh-et-al,Clifton-Huang,IKV23skew,LinialR05,saks1993,S-Wigderson-22,Saxton-2013,YY24Essential}. In particular, the problem of determining the minimum possible size of a \emph{skew} cover of the $n$-dimensional hypercube received a lot of attention~\cite{alon88balance,IKV23skew,klein2022slicing,saks1993,essential2024,yehuda2021slicing}. A collection $\cH$ of hyperplanes in $\R^n$ covering $\bin^n$ is called a \emph{skew cover} of $\bin^n$, if for each hyperplane $H\in \cH$, all of the $n$ coordinate entries of the normal vector of $H$ are non-zero. In other words, for each of the hyperplanes $H\in \cH$, when writing down the corresponding hyperplane equation, all $n$ variables must appear with non-zero coefficients. Geometrically, this means that all hyperplanes in $\cH$ are ``skew'' in the sense of not being parallel to any of the coordinate axes.

The problem of estimating the smallest possible number of hyperplanes in a skew cover of $\bin^n$ has been discussed extensively in a survey paper by Saks~\cite[Problem~3.63]{saks1993} from 1993, and as explained in~\cite[p.~241]{saks1993}, it is not hard to show that every skew cover of $\bin^n$ must consist of at least $\Omega(\sqrt{n})$ hyperplanes. Yehuda and Yehudayoff~\cite{yehuda2021slicing} improved this lower bound to $\Omega(n^{0.51})$ via a reduction from this problem to a problem about hyperplane collections slicing all edges of the hypercube. Later, Klein~\cite{klein2022slicing} improved the lower bound for this slicing problem, implying a lower bound of $\Omega(n^{2/3}/(\log n)^{4/3})$ for the skew covers problem via the same reduction. Recently, the authors~\cite{essential2024} observed that a result of Linial and Radhakrishnan~\cite{LinialR05} from 2005 implies that any skew cover  of $\bin^n$ must consist of at least $n/2$ hyperplanes. The same observation was independently obtained by Ivanisvili, Klein, and Vershynin~\cite{IKV23skew} shortly afterwards. This lower bound of $n/2$ resolves the problem for skew covers up to constant factors. Indeed, it is easy to construct skew covers  of $\bin^n$ consisting of $n$ hyperplanes.

In this short note, we generalize the notion of a skew cover by considering the following weaker nondegeneracy condition: We consider hyperplane collections $\cH$ with the condition that for every $v\in \bin^n$ and every coordinate direction $i=1,\dots,n$, there exists a hyperplane $H\in \cH$ with $v\in H$  such that the $i$-th coordinate of the normal vector of $H$ is non-zero. This condition means that when expressing $H$ via its hyperplane equation $\ang{a,x}=b$ (where $a\in \R^n$ is the normal vector of $H$, and $b\in \R$), then the variable $x_i$ appears with a non-zero coefficient in this equation. Geometrically, the condition on $\cH$ can be rephrased as saying that for every vertex $v\in \bin^n$ of the hypercube, and every hypercube edge $e$ incident to $v$, there exists a hyperplane $H\in \cH$ that contains $v$ but not the entire edge $e$.

We show the following lower bound for the number of hyperplanes in such a nondegenerate cover.

\begin{theorem}\label{thm:main}
    Let $\cH$ be a collection of hyperplanes in $\R^n$ satisfying the following condition: for every $v\in \bin^n$ and every $i=1,\dots,n$, there exists a hyperplane $H\in \mathcal{H}$ through $v$ such that the $i$-th coordinate of the normal vector of $H$ is non-zero. Then we must have $|\cH|\ge n/2$.
\end{theorem}

Note that the bound of $n/2$ is tight up to constant factors. Indeed, for any $n\ge 2$, one can consider the $n$ hyperplanes given by the equations $x_1+\dots+x_{n-1}-(n-1)x_n=0$ and $x_1+\dots+x_n=t$ for $t=1,\dots,n-1$. Also note that \cref{thm:main} generalizes the above-mentioned lower bound results for skew covers, since every skew cover of $ \bin^n$ automatically satisfies the condition in \cref{thm:main}.

As mentioned above, the problem of determining the minimum size of a skew cover of the hypercube $\bin^n$ is related to the problem of determining the minimum number of hyperplanes required to slice every edge of the hypercube. We say that a hyperplane $H$ in $\R^n$ slices an edge of the $n$-dimensional hypercube if $H$ contains exactly one interior point of the edge. Equivalently, when denoting the two endpoints of the edge as $v,v'\in \bin^n$, and writing $H = \{x\in \R^n\mid \ang{a,x} = b \}$, the hyperplane $H$ slices the edge $vv'$ if $\ang{a,v} - b$ and $\ang{a,v'} - b$ are both non-zero and have different signs.

Now, it is natural to ask how many hyperplanes in $\R^n$ are needed in order to slice every edge of the hypercube. More formally, what is the minimum possible size of a collection $\cH$ of hyperplanes in $\R^n$ such that every edge of the $n$-dimensional hypercube $[0,1]^n$ is sliced by at least one hyperplane in $\cH$? This problem has attracted the attention of many researchers over the past 50 years~\cite{AHLSWEDE1990137,alon88balance,EmamyKhansary86,GotsmanLinial94,grunbaum72,klein2022slicing,ONeil71,paturi1990threshold,saks1993,yehuda2021slicing}, and has applications to the study of perceptrons~\cite{ONeil71} and of threshold circuits for parity~\cite{paturi1990threshold,yehuda2021slicing}.  O'Neil showed that any  such collection $\cH$ of hyperplanes must have size $|\cH|\ge \Omega(\sqrt{n})$, which was later improved by Yehuda and Yehudayoff~\cite{yehuda2021slicing} to $|\cH|\ge \Omega(n^{0.51})$. This was further improved by Klein~\cite{klein2022slicing} to $|\cH|\ge \Omega(n^{2/3}/(\log n)^{4/3})$, and the current best lower bound is $\Omega(n^{13/19}/\log^{-32/19})$ due to the authors~\cite{sauermann2025improvedlowerboundhypercube}. On the other hand, the best known upper bound is $\lceil (5/6)n\rceil$, given by a construction of Paterson (see~\cite{saks1993}).

It is conjectured that the minimum possible size of a collection $\cH$ of hyperplanes such that every edge of the $n$-dimensional hypercube is sliced by one of these hyperplanes is on the order of $n$. So far, this conjecture is only known to be true under some very specific assumptions on the form of the hyperplanes in $\cH$. It is not hard to show that at least $n$ hyperplanes are needed if all of the hyperplanes can be described by equations in which all variables have non-negative coefficients (or, equivalently, if one can choose normal vectors for the hyperplanes, such that all entries of all of these normal vectors are non-negative), see~\cite{AHLSWEDE1990137,GotsmanLinial94}.

The conjecture that $\Omega(n)$ hyperplanes are needed in order to slice all edges of the hypercube is also known in the case where all hyperplanes in $\cH$ have normal vectors in $\{1,-1\}^n$. Equivalently, this condition means that all hyperplanes can be expressed in such a way that all variables appear with coefficients $1$ or $-1$ in all of the hyperplane equations. Paturi and Saks observed that a result of Alon, Bergmann, Coppersmith, and Odlyzko~\cite{alon88balance} implies a lower bound of $|\cH|\ge n/2$ for the size of $\cH$ in this case (Alon--Bergmann--Coppersmith--Odlyzko~\cite{alon88balance} briefly remarked that their work is related to this question, and the arguments for the deduction can be found in~\cite[Proposition~3.75]{saks1993}).

As a corollary of \cref{thm:main}, we can show that the conjecture holds for a significantly wider class of hyperplanes. More specifically, we show that $\Omega(n)$ hyperplanes are needed to slice every edge of the $n$-dimensional hypercube, if  each of the hyperplanes has a normal vector in $\{-C, \dots, C\}^n$ (for some fixed positive integer $C$). In other words, this condition means that all of the hyperplanes can be described by equations in which all variables appear with bounded integer coefficients. 

\begin{corollary}\label{cor:slicing}
    Let $C$ be a positive integer and let $\cH $ be a collection of hyperplanes in $\mathbb{R}^n$ with normal vectors in $\{-C, \dots, C\}^n$, such that every edge of the $n$-dimensional hypercube $[0,1]^n$ is sliced by a hyperplane in $\cH$. Then $|\cH|\ge n/(4C)$.
\end{corollary}

Note that even the case $C=1$ of \cref{cor:slicing} already extends the class of hyperplanes with normal vectors in $\{1,-1\}^n$ for which the conjecture was previously known. Indeed, for the arguments in~\cite{alon88balance}, it is crucial that none of the hyperplane normal vectors contains any zero entries. In contrast, in our result, zero entries are allowed as well in the normal vectors of the hyperplanes.

Also note that in the problem of determining the minimum possible size of a collection of hyperplanes slicing all edges of the $n$-dimensional hypercube, one may assume without loss of generality that all hyperplanes are described by equations with integer coefficients. Indeed, given any collection $\cH$ of hyperplanes slicing all hypercube edges, one can first ``wiggle'' the hyperplanes in such a way that all coefficients in the hyperplanes become rational and then multiply the equations by the respective common denominators to obtain integer coefficients. Thus, the general problem can be reduced to the case of integer coefficients. Our result in \cref{cor:slicing} resolves the problem up to constant factors in the case of \emph{bounded} integer coefficients.

\medskip

\noindent \emph{Notation.} For a vector $a = (a_1,\dots, a_n)\in \R^n$, the support of $a$ is denoted as $\supp(a) := \{i \mid a_i\ne 0\}$.

\medskip

\noindent \textit{Acknowledgements.} We would like to thank Ting-Wei Chao for raising the question to prove an $\Omega(n)$ lower bound for the size of a collection of hyperplanes with normal vectors in $\{-1,0,1\}^n$ slicing every edge of the $n$-dimensional hypercube, which led to this note. We would also like to think the referees for carefully reading our note and providing helpful comments.

\section{Proofs}

We first present the proof of \cref{thm:main}. It is partly inspired by arguments of Linial and Radhakrishnan~\cite{LinialR05} for their lower bound for the sizes of so-called essential covers of the hypercube. However, several new ideas are required in our setting.

Like Linial and Radhakrishnan~\cite{LinialR05}, we use the following well-known result of Alon and F\"uredi~\cite{Alon1993Covering} on collections of hyperplanes covering every vertex of the hypercube except for the origin.

\begin{theorem}[\cite{Alon1993Covering}]\label{thm:alon-furedi}
        Suppose $\cH$ is a collection of hyperplanes in $\mathbb{R}^s$ covering every vertex of the hypercube $\bin^s$ except for the origin (and not covering the origin). Then we must have $|\cH|\ge s$.
    \end{theorem}

This theorem can easily be proved via the Combinatorial Nullstellensatz (see~\cite[Theorem~6.3]{alon-nullstellensatz99}). Let us now prove \cref{thm:main}.

\begin{proof}[Proof of \cref{thm:main}]
    Let $\cH$ be a collection of hyperplanes satisfying the condition in the theorem statement. For every $v\in \bin^n$, let $\cH_v = \{H\in \cH\mid v\in H\} \subseteq \cH$ denote the set of hyperplanes in $\cH$ passing through $v$. Now, let $w\in \bin^n$ be a vertex minimizing $|\cH_w|$. We may assume without loss of generality that $w=0$ (otherwise we may consider a change of variables replacing $x_i$ by $1-x_i$ in all of the hyperplane equations for all indices $i$ with $w_i = 1$). This means that we have $|\cH_v|\ge |\cH_w|=|\cH_0|$ for all $v\in \bin^n$.

    Let $H_1,\dots, H_m$ be the hyperplanes in $\cH_w =\cH_0$. For $j=1,\dots,m$, let $a^{(j)}$ be the normal vector of the hyperplane $H_j$, and consider the support $\supp(a^{(j)})$ of the vector $a^{(j)}$. By the assumption in the theorem statement (applied to $v=0$), we know that for every $i=1,\dots,n$, there exists $j\in \{1,\dots, m\}$ with $i\in \supp(a^{(j)})$. Thus, we have $\bigcup_{j=1}^m \supp(a^{(j)})= \{1,\dots, n\}$.
    
    Now, for each $j=1,\dots,m$, let us define $T_j = \supp(a^{(j)})\setminus \bigcup_{h=1}^{j-1} \supp(a^{(h)})$. In other words, $T_j$ is the subset of all indices $i\in\{1,\dots,n\}$ such that $a^{(j)}$ is the first vector in the sequence $a^{(1)},\dots,a^{(m)}$ whose $i$-th coordinate is non-zero. Note that $T_1,\dots, T_m$ form a partition of $\{1,\dots, n\}$ (also note that some of the sets $T_1,\dots, T_m$ may be empty).
    
    For each $j=1,\dots,m$, and every $i\in T_j\subseteq  \supp(a^{(j)})$, the $i$-th coordinate $a^{(j)}_{i}$ of the vector $a^{(j)}$ is either positive or negative. By the pigeonhole principle, for each $j=1,\dots,m$, there exists a subset $T_j'\subseteq T_j$ with $|T_j'|\ge |T_j|/2$ such that $a^{(j)}_{i}$ has the same sign for all $i\in T_j'$. Let us now define 
     \[S = \bigcup_{j=1}^m T_j',\] 
     and note that $|S| = \sum_{j=1}^{m} |T_j'|\ge \sum_{j=1}^{m} |T_j|/2=n/2$.

    We now consider the $|S|$-dimensional subcube $Q_S\subseteq \bin^n$ given by 
    \[Q_S = \{v\in \bin^n \mid \supp(v)\subseteq S\}.\]
    In other words, $Q_S$ is formed by taking all vertices $v\in \bin^n$ with $v_i=0$ for $i\in \{1,\dots,n\}\setminus S$. We show the following key claim.
    \begin{claim}\label{claim:Qs}
        For each $v\in Q_S$ with $v\ne 0$, there exists a hyperplane $H\in \cH_0$ with $v\not\in H$. 
    \end{claim}
    \begin{proof}
        Recall that $\cH_0 = \{H_{1}, \dots, H_{m}\}$, and let $j\in \{1,\dots,m\}$ be the smallest index with $\supp(a^{(j)})\cap \supp(v)\ne \varnothing$ (such an index exists as $\supp(v)\ne \varnothing$ and $\bigcup_{j=1}^m \supp(a^{(j)})= \{1,\dots, n\}$). Then we have $\supp(a^{(j)})\cap \supp(v)\subseteq \supp(a^{(j)})\setminus \bigcup_{h=1}^{j-1} \supp(a^{(h)})=T_j$. Furthermore, observe that $\supp(a^{(j)})\cap \supp(v)\subseteq \supp(v)\subseteq S$. Thus, we can conclude $\supp(a^{(j)})\cap \supp(v)\subseteq T_j\cap S= T_j'$ by the construction of $S$. This means that $a^{(j)}_{i}$ has the same sign for all $i\in \supp(a^{(j)})\cap \supp(v)\subseteq T_j'$. Therefore we must have $\ang{a^{(j)}, v}\ne 0$, as every term in the sum 
        \[\ang{a^{(j)},v} = \sum_{i\in \supp(a^{(j)})\cap \supp(v)}a^{(j)}_iv_i=\sum_{i\in \supp(a^{(j)})\cap \supp(v)}a^{(j)}_i\]
        has the same sign and $\supp(a^{(j)})\cap \supp(v)\ne \varnothing$. On the other hand, the hyperplane $H_j$ is described by the equation $\ang{a^{(j)},x}=0$ (recalling that $H_j\in \cH_0$ passes through $0$). Therefore the hyperplane $H_{j}\in \cH_0$ does not contain $v$.
    \end{proof}

    As a consequence of \cref{claim:Qs}, we have the following.

    \begin{claim}\label{claim:subcube-cover}
        $\cH\setminus \cH_0$ is a set of hyperplanes covering $Q_S\setminus\{0\}$ and not covering $0$.
    \end{claim}

    \begin{proof}
        Suppose for contradiction that there exists a vertex $v\in Q_S$ with $v\ne 0$ which is not contained in any of the hyperplanes in $\cH\setminus \cH_0$. By \cref{claim:Qs}, there is a hyperplane in $\cH_0$ not containing $v$. Thus, $v$ is contained in at most $|\cH_0| - 1$ hyperplanes in $\cH$, meaning that $|\cH_v| \le |\cH_0| - 1$. This is a contradiction to the minimality of $|\cH_w|=|\cH_0|$, so every vertex $v\in Q_S\setminus \{0\}$ is covered by $\cH\setminus \cH_0$. On the other hand, by definition of $\cH_0$, none of the hyperplanes in $\cH\setminus \cH_0$ passes through $0$.
    \end{proof}

    Now we can finish the proof by applying Theorem \ref{thm:alon-furedi}. Let us consider the $|S|$-dimensional linear subspace $U\subseteq \mathbb{R}^n$ given by $U=\{x\in \mathbb{R}^n \mid \supp(x)\subseteq S\}$, and note that $Q_S=\bin^n\cap U$. Note that for each $H\in \cH\setminus \cH_0$, we obtain a hyperplane $H\cap U$ in $U$ (indeed, as $0\not\in H$, we have $H\cap U\ne U$).  Thus, taking $H\cap U$ for every $H\in \cH\setminus \cH_0$, by  \cref{claim:subcube-cover} we obtain a collection of hyperplanes in $U$ covering all points in $Q_S\setminus\{0\}$ and not covering $0$. Noting that $U\cong \mathbb{R}^{|S|}$ (and $Q_S\cong \bin^{|S|}$ under this isomorphism), by \cref{thm:alon-furedi} we have $|\cH\setminus \cH_0|\ge |S|\ge n/2$. Therefore we can conclude that $|\cH| \ge |\cH\setminus \cH_0|\ge n/2$, as desired.
\end{proof}

\begin{remark}
    In \cref{thm:main}, one can in fact obtain the stronger bound $|\cH|\ge (n/2)+1$ via the same proof. Indeed, in the proof we have $\cH_0\ne \emptyset$ and we have shown $|\cH\setminus \cH_0|\ge |S|\ge n/2$, yielding $|\cH| =|\cH\setminus \cH_0|+|\cH_0|\ge (n/2)+1$.
\end{remark}

Let us now show how \cref{cor:slicing} can be deduced from \cref{thm:main}.

\begin{proof}[Proof of \cref{cor:slicing}]
Let $\cH$ be a set of hyperplanes in $\mathbb{R}^n$ with normal vectors in $\{-C,\dots, C\}^n$ such that every edge of the $n$-dimensional hypercube is sliced by at least one of the hyperplanes in $\cH$. Consider the set of hyperplanes $\cH'$ in $\mathbb{R}^n$ constructed from $\cH$ as follows: For each hyperplane $H\in \cH$ with hyperplane equation $\ang{a,x} = b $ and normal vector $a\in \{-C,\dots, C\}^n$, let $\cS_H$ denote the set of the $2C$ hyperplanes with hyperplane equations $\ang{a,x} = \floor{b} + z$ for $z\in \{-(C-1),\dots, C\}$. Let $\cH' = \bigcup_{H\in \cH} \cS_H$, and note that then we have $|\cH'| \le 2C\cdot |\cH|$.

\begin{claim}\label{claim:reduction}
    $\cH'$ is a collection of hyperplanes satisfying the condition in \cref{thm:main}.
\end{claim}

\begin{proof}
    Let $v\in \bin^n$ be a vertex and let $i\in \{1,\dots,n\}$. There must be a hyperplane $H\in \cH$ slicing the edge adjacent to $v$ along the $i$-th coordinate direction. Let $v'$ be the other endpoint of this edge (then $v$ and $v'$ only differ in the $i$-th coordinate). Let $H$ be described by the equation $\ang{a,x} = b$  with $a\in \{-C,\dots, C\}^n$ and $b\in \R$. Since $H$ slices the edge $vv'$ of the hypercube, the vertices $v$ and $v'$ lie on different sides of the hyperplane $H$, so $\ang{a,v} - b$ and $\ang{a,v'} - b$ have different signs (meaning one of the numbers $\ang{a,v} - b$ and $\ang{a,v'} - b$ is positive, while the other one is negative). 
    On the other hand, $v$ and $v'$ differ only in the $i$-th coordinate (and the $i$-th coordinate of one of these vertices is $1$, while for the other one it is $0$), so we have $|\ang{a,v}-\ang{a,v'}|=|a_i|\le C$. Therefore we obtain
    \[|(\ang{a,v} - b) - (\ang{a,v'} - b)| = |\ang{a,v}-\ang{a,v'}| =|a_i| \le C,\]
    and since $\ang{a,v} - b$ and $\ang{a,v'} - b$ have different signs, we can conclude that $a_i\ne 0$ and $|\ang{a,v} - b|<C$ (indeed, if $|\ang{a,v} - b|\ge C$, then, due to the inequality above,  $\ang{a,v'} - b$ would need to be zero or have the same sign as $\ang{a,v} - b$). Thus, we have $\ang{a,v}\in (b-C,b+C)$, and since $a$ and $v$ are integer vectors (and $C$ is a positive integer), this implies
    \[\ang{a,v}\in (b-C,b+C)\cap \mathbb{Z}\subseteq \{\floor{b-C}+1,\dots, \floor{b+C}\}=\{\floor{b}-(C-1),\dots, \floor{b}+C\}.\]
    This means that $\ang{a,v}=\floor{b}+z$ for some $z\in \{-(C-1),\dots, C\}$, and so $v$ is contained in one of the hyperplanes in $\cS_H\subseteq \cH'$. Each of these hyperplanes has normal vector $a$, and the $i$-th coordinate of $a$ is $a_i\ne 0$. Thus, $v$ is indeed contained in a hyperplane in $\cH'$ whose normal vector has a non-zero entry in the $i$-th coordinate.
\end{proof}

By \cref{thm:main}, we have $|\cH'| \ge n/2$. Since $|\cH'|\le 2C\cdot |\cH|$ by construction, we can conclude that $|\cH|\ge n/(4C)$, as desired.
\end{proof}

\printbibliography

\end{document}